\documentclass[12pt, reqno]{amsart}
\usepackage{amsmath,amsopn,amssymb,amsthm}
\usepackage[english]{babel}
\usepackage[backref=page, breaklinks=true,colorlinks=true,linkcolor=blue,citecolor=blue,urlcolor=blue]{hyperref}

\usepackage{graphicx}

\renewenvironment{proof}[1][Proof]{\textbf{#1.} }{\ \rule{0.5em}{0.5em}}

\renewenvironment{proof}[1][Proof]{\textbf{#1.} }
{\ \rule{0.5em}{0.5em}}
\newtheorem{theorem}{Theorem}
\newtheorem{prop}{Proposition}
\newtheorem{lemma}{Lemma}

\theoremstyle{definition}

\newtheorem{remark}{Remark}

\newtheorem{problem}{Problem}

\begin{document}

\title
[Extremal problems for convex curves with \dots]
{Extremal problems for convex curves with \\ given relative Chebyshev radius}
\author{Vitor~Balestro, Horst~Martini, Yurii~Nikonorov, Yulia~Nikonorova}

\address{Vitor Balestro \newline
Instituto de Matem\'{a}tica e Estat\'{i}stica \newline
Universidade Federal Fluminense \newline
24210201 Niter\'{o}i \newline
Brazil}
\email{vitorbalestro@id.uff.br}

\address{Horst Martini \newline
Fakult\"{a}t f\"{u}r Mathematik \newline
Technische Universit\"{a}t Chemnitz \newline
09107 Chemnitz \newline
Germany}
\email{martini@mathematik.tu-chemnitz.de}

\address{Yurii Nikonorov \newline
Southern Mathematical Institute of \newline
the Vladikavkaz Scientific Center of \newline
the Russian Academy of Sciences, \newline
Vladikavkaz, Markus st., 22, \newline
362027, Russia}
\email{nikonorov2006@mail.ru}

\address{Yulia Nikonorova\newline
Volgodonsk Engineering Technical Institute the branch \newline
of National Research Nuclear University ``MEPhI'',\newline
Rostov region, Volgodonsk, Lenin st., 73/94, \newline
347360,  Russia}
\email{nikonorova2009@mail.ru}

\begin{abstract}
The paper is devoted to some extremal problems for convex curves and polygons in the Euclidean plane referring to
the relative Chebyshev radius. In particular, we determine the relative Chebyshev radius for an arbitrary triangle. Moreover, we derive
the maximal possible perimeter for  convex curves and convex $n$-gons of a given relative Chebyshev radius.

\vspace{2mm}

\vspace{2mm} \noindent Key words and phrases: approximation by polytopes, convex curve, convex polygon, relative Chebyshev radius.
\end{abstract}

\maketitle

\section{Introduction}

Let $(X,d)$ be a bounded metric space. Let us consider the metric invariant
\begin{equation}\label{chebir1}
\delta(X)=\inf\limits_{p\in X}\, \sup\limits_{q\in X} \,d(p,q)
\end{equation}
arising in approximation theory, where it is called {\it the relative Chebyshev radius of $X$ with respect to $X$ itself},
see, e.g., \cite[p.~119]{Am1986} and \cite{AmZig1980}.
It has also the following obvious geometric sense for compact $X$: $\delta(X)$ is the smallest radius of a ball having its center in $X$ and covering $X$.
For brevity we will call $\delta(X)$ just the relative Chebyshev radius of $X$.

The study of extremal problems for convex curves in the Euclidean plane with a given relative Chebyshev radius started with the paper
\cite{Walter2017} by Rolf~Walter.
In particular, he conjectured that $L(\Gamma)\geq \pi \cdot \delta(\Gamma)$ for any closed convex curve $\Gamma$ in the Euclidean plane,
where $L(\Gamma)$ is the length of $\Gamma$ and $d$ is the standard restricted Euclidean metric.
In~\cite{Walter2017}, this conjecture is proved for the case that
$\Gamma$ is a convex curve of class $C^2$ and all curvature centers of $\Gamma$ lie in the interior of $\Gamma$.
It is also shown that the equality $L(\Gamma)=\pi \cdot \delta(\Gamma)$ in this case
holds if and only if $\gamma$ is of constant width.

It is also proved in \cite{Walter2017} that all $C^2$-smooth convex curves have good approximations by polygons in the sense of the relative Chebyshev radius \eqref{chebir1}.
This observation leads to natural extremal problems for convex polygons.
In particular, the following result given in \cite{Walter2017} holds.

\begin{theorem}[\cite{Walter2017}]\label{theo2wal}
For each triangle  $P$ in the Euclidean plane,  one has
$$
L(P)\geq 2\sqrt{3}\,\cdot \delta(P),
$$
with equality exactly for equilateral triangles.
\end{theorem}

One of the goals of the present paper is to simplify the proof of this theorem. For this aim we apply the explicit expression for
the value $\delta(P)$ obtained in the following.

\begin{theorem}\label{theoechrt}
Let $P$ be a triangle in the Euclidean plane with side lengths $a\geq b \geq c$ and with angles $\alpha \geq \beta \geq \gamma$.
Then the following formula holds:
$$
\setlength{\jot}{200pt}
\delta(P)=
\begin{cases}
\displaystyle{\frac{a}{2}} & {for} \, \alpha \geq \pi/2,
\vspace{1mm}
\\
\vspace{1mm}
b\sin(\gamma) & \mbox{for}\,\,\, \gamma \geq \pi/4, \\
\displaystyle{\frac{b}{2\cos(\gamma)}} &\mbox{for}\,\, \gamma \leq \pi/4 \mbox{  and  } \alpha \leq \pi/2\,.
\end{cases}
$$
\end{theorem}

This theorem is proved in Section \ref{sect.2} of the paper. In the same section, we give a new proof of Theorem \ref{theo2wal}.
We also note that in Section \ref{sect.1} we fix the notation and consider several auxiliary results.
Section \ref{sect.3} is devoted to convex curves (polygons) of maximal perimeter among
all curves (among all polygons with a given number of vertices) with a fixed value of
the relative Chebyshev radius \eqref{chebir1}. The corresponding results are contained in
Theorem \ref{maxgeneral} and Theorem \ref{maxpolygon}.
Finally, we briefly discuss some related unsolved problems.

\smallskip

\section{Notation and some auxiliary results}\label{sect.1}

We identify the Euclidean plane with $\mathbb{R}^2$ supplied with the standard Euclidean metric~$d$, where $d(x,y)=\sqrt{(x_1-y_1)^2+(x_2-y_2)^2}$.

We call $\Gamma$ {\it a convex curve} if it is the boundary of some convex compact set in the Euclidean plane $\mathbb{R}^2$.
Important examples of convex curves are {\it  convex polygons {\rm(}convex closed polygonal chains}).
A polygon $P$ is called and $n$-gon if it has exactly $n$ vertices.
For $n=2$ we get line segments. The perimeter $L(P)$ of any $2$-gon
is defined as the double length of the line segment $P$. Such definition is justified,
since it leads to the continuity of the perimeter as a functional on the set of convex polygons with respect to the Hausdorff distance.

Let $\Gamma \subset \mathbb{R}^2$ be a compact set (in particular, a convex curve). We define the function $\mu:\Gamma \rightarrow \mathbb{R}$ as follows:
\begin{equation}\label{mufunk}
\mu (x)=\max\limits_{y \in \Gamma} d(x,y).
\end{equation}

If a point $x\in \Gamma$ is such that $\mu(x)=\min\limits_{y \in \Gamma} \mu(y)=\delta(\Gamma)$ (see \eqref{chebir1}),
then we call it {\it extremal} (for $\delta(\Gamma)$), whereas any point
$x_0\in \Gamma$ with $d(x,x_0)=\mu(x)=\max\limits_{y \in \Gamma} d(x,y)$ is called
{\it a footpoint} (for~$x$) and the corresponding chord $[x,x_0]$ is called {\it distinguished}.
\smallskip

For any polygon $P$ and any given point~$x\in P$,  $\max\limits_{y \in P} d(x,y)$ is achieved at a vertex of $P$ (see, e.g., Lemma 4.1 in  \cite{Walter2017}),
i.e., $\mu(x)$ is equal to the maximal distance from $x$ to vertices of $P$.
Note also that the diameter $D:=\max\limits_{x,y \in P} d(x,y)$  of a polygon $P$ always connects two vertices.

\smallskip

The following property (monotonicity of the perimeter) of convex curves is well known (see, e.g., \cite[\S 7]{BoFe1987}).

\begin{prop}\label{monotper}
If a convex curve $\Gamma_1$ is inside another convex curve $\Gamma_2$ in the Euclidean plane,
then the perimeter of $\Gamma_1$ is less or equal to the perimeter of $\Gamma_2$, and equality holds if and only if $\Gamma_1=\Gamma_2$.
\end{prop}

The following simple result is very useful.

\begin{prop}\label{semcircle}
Let $\Gamma$ be a convex curve in $\mathbb{R}^2$ such that $\Gamma$ contains the line segment $[p,q]$ with the property
$\Gamma \subset \left\{ x\in \mathbb{R}^2\,|\, d(x,o)\leq \frac{1}{2}d(p,q) \right\}$, where $o$ is the  midpoint of $[p,q]$.
Then $\delta(\Gamma)= \frac{1}{2}d(p,q)=\mu(o)$. Moreover, $o$ is a unique extremal point for $\delta(\Gamma)$.
\end{prop}

\begin{proof} Since $\mu(o)=\max\limits_{x \in \Gamma} d(o,x) \leq \frac{1}{2}d(p,q)$ and $d(o,p)=d(o,q)=\frac{1}{2}d(p,q)$, then
$\mu(o)=\frac{1}{2}d(p,q)$. On the other hand, for any $x \in \Gamma$ we have $d(x,p)+d(x,q) \geq d(p,q)$, and therefore,
$\mu(x)=\max\limits_{y \in \Gamma} d(x,y) \geq \max \{d(x,p),d(x,q)\}\geq \frac{1}{2}d(p,q)$.
Hence, $o$ is an extremal point for $\delta(\Gamma)$ and $\delta(\Gamma)=\mu(o)= \frac{1}{2}d(p,q)$.
It is also clear that $o$ is a unique extremal point, since $\max \{d(x,p),d(x,q)\}> \frac{1}{2}d(p,q)$ for any $x \neq o$.
\end{proof}

\section{The relative Chebyshev radius for triangles}\label{sect.2}

In this section we deal with triangles in Euclidean plane $\mathbb{R}^2$.

\begin{lemma}\label{lem1}
Let $P$ be a triangle $KLM$ such that $\angle KLM \geq \pi/4$ and $\angle LKM \geq \pi/4$.
Then the minimal value of $\mu(x)=\max\limits_{y \in P} d(x,y)$ for $x$ lying in the line segment $[K,L]$
is reached exactly at the point $N\in [K,L]$, such that the line $MN$ is orthogonal to the line $KL$, and $d(M,N)$ is equal to
the length of the altitude of the triangle $KLM$ through the vertex $M$.
\end{lemma}

\begin{proof}
It is clear that $d(M,N) \geq d(N,K)$ and $d(M,N) \geq d(N,L)$ (note that $\angle NLM \geq \pi/2-\angle NLM=\angle LMN$ and
$\angle NKM \geq \pi/2-\angle NKM=\angle KMN$), see Fig. \ref{Fig1}\,a).
Therefore,
$$
\mu(N)=\max\{d(N,K), d(N,L), d(N,M)\}=d(N,M).
$$
If $x\in [K,L]$ and $x\neq N$, then $d(M,x) >  d(M,N)$ and, therefore,
$$
\mu(x)=\max\{d(x,K), d(x,L), d(x,M)\}> d(M,N),
$$
as required.
\end{proof}
\smallskip

Let $ABC$ be a triangle with vertices $A,B,C$. We put $a=d(B,C)$, $b=d(A,C)$, $c=d(A,B)$, where $a\geq b \geq c$,
and $\alpha=\angle CAB$, $\beta=\angle ABC$, $\gamma =\angle BCA$. It is clear that $\alpha \geq \beta \geq \gamma$.

We are going to calculate the value of $\delta(P)$.

\begin{lemma}\label{lem2}
For a triangle $P$ with $\alpha \geq \pi/2$, we have $\delta(P)=a/2$ with a unique extremal point $o$, the midpoint of $[B,C]$.
\end{lemma}

\begin{proof}
Since $\alpha \geq \pi/2$, we have
$P \subset \left\{ x\in \mathbb{R}^2\,|\, d(x,o)\leq a/2 \right\}$, where $o$ is the  midpoint of $[B,C]$.
Now it suffices to apply Proposition \ref{semcircle}.
\end{proof}

\begin{lemma}\label{lem3}
For a triangle $P$ with $\gamma \geq \pi/4$, the value $\delta(P)$ is equal to $b\sin(\gamma)=c\sin(\beta)$,
the length of the altitude of $P$ through the vertex $A$.
\end{lemma}

\begin{proof}
It is clear that $\min\limits_{x\in P} \mu(x)=\min \{M_1, M_2, M_3\}$,
where $M_1:=\min\limits_{x\in [A,B]} \mu(x)$, $M_2:=\min\limits_{x\in [A,C]} \mu(x)$, and $M_3:=\min\limits_{x\in [B,C]} \mu(x)$ (see \eqref{mufunk}).
By Lemma \ref{lem1}, $M_1$, $M_2$, $M_3$ are equal to the length of the altitude of the triangle $P$ through the vertices $C$, $B$, $A$ respectively.
Now it is clear (recall that $a\geq b \geq c$) that $\delta(P)=\min\limits_{x\in P} \mu(x)=M_3$, the length of the altitude of $P$ through the vertex $A$.
\end{proof}

\begin{remark}\label{rem1}
If $\gamma \geq \pi/4$, then $\alpha \geq \beta \geq \gamma \geq \pi/4$ and $\alpha \leq \pi -\beta-\gamma \leq \pi/2$.
Moreover, $\alpha=\pi/2$ if and only if $\beta=\gamma=\pi/4$.
\end{remark}

\begin{figure}[t]
\begin{minipage}[h]{0.35\textwidth}
\center{\includegraphics[width=0.99\textwidth]{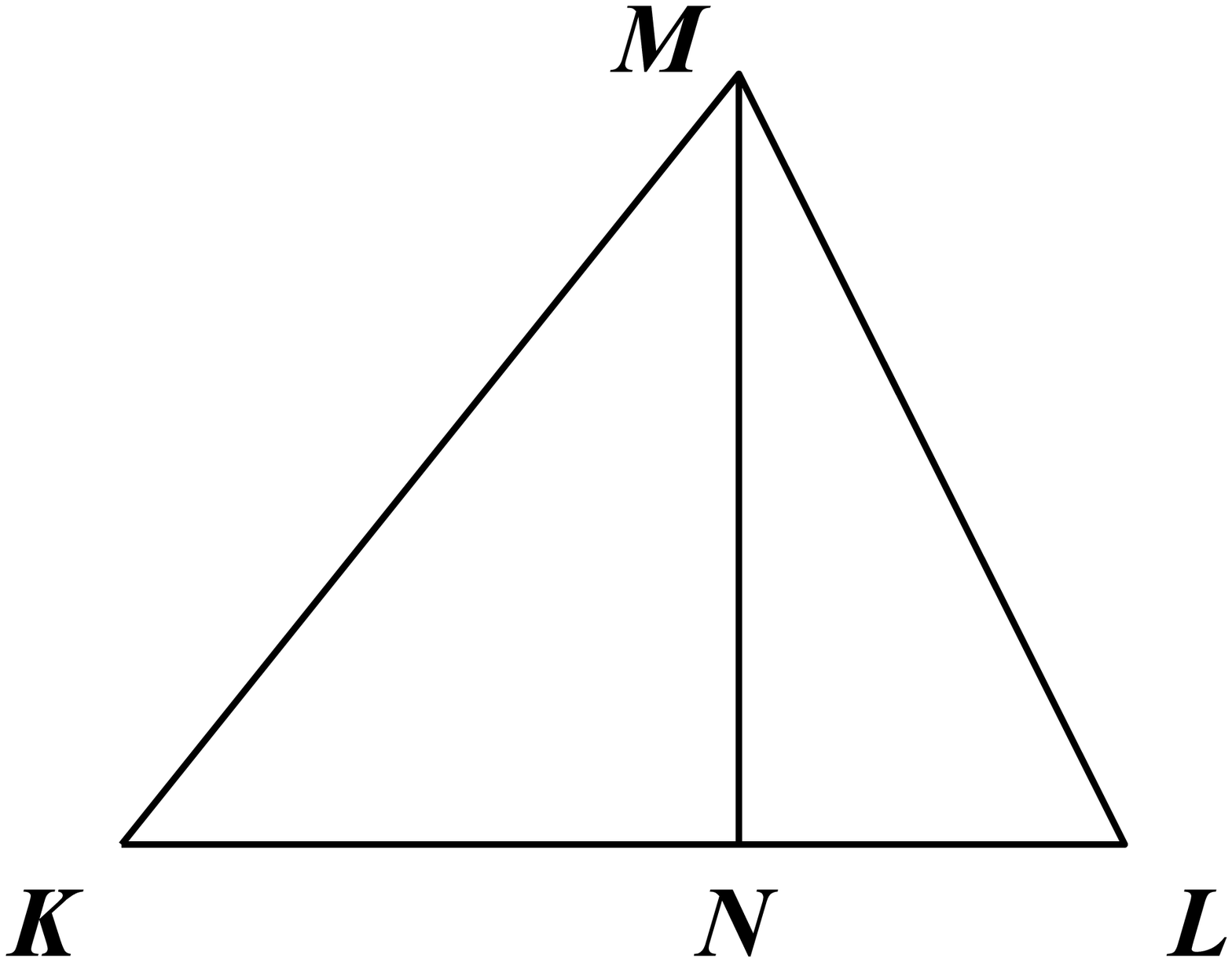}}    a) \\
\end{minipage}
\quad\quad
\begin{minipage}[h]{0.45\textwidth}
\center{\includegraphics[width=0.99\textwidth]{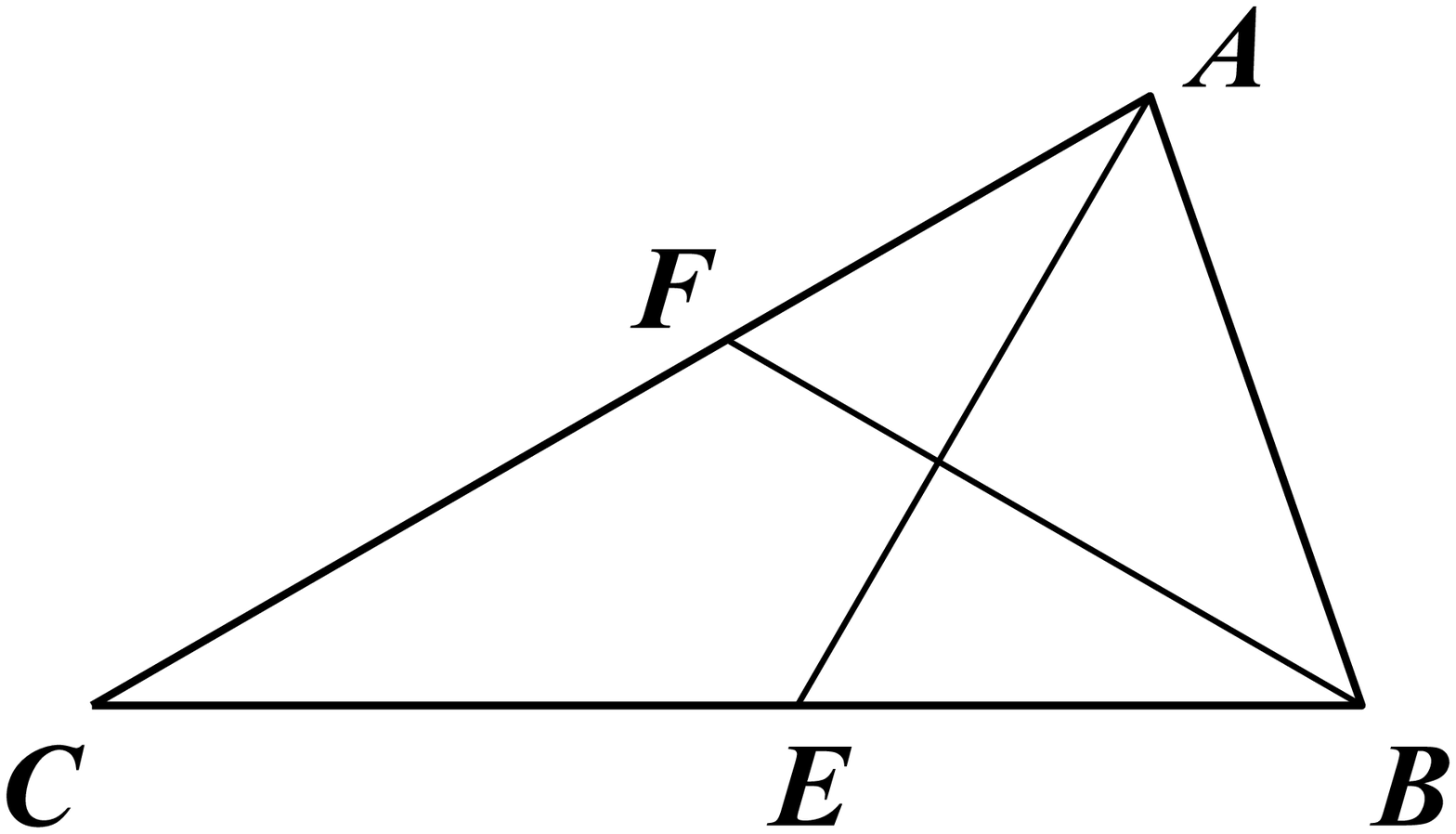}} b) \\
\end{minipage}
\caption{The pictures for: a) Lemma \ref{lem1}; b) Lemma \ref{lem4}.}
\label{Fig1}
\end{figure}

\begin{lemma}\label{lem4}
For a triangle $P$ with $\gamma \leq \pi/4$ and $\alpha \leq \pi/2$, we have $\delta(P)=\frac{b}{2\cos(\gamma)}$.
\end{lemma}

\begin{proof} Let us consider the points $E\in [B,C]$ and $F=[A,C]$ such that $\angle EAC= \angle FBC =\gamma$, see Fig. \ref{Fig1}\,b).
Note that $d(E,C)=d(E,A) \geq d(E,B)$ (the midpoint of $[B,C]$ is on the line segment $[E,C]$ due to the inequality $\angle CAB =\alpha \leq \pi/2$).
Hence $\mu(E)=d(E,C)=d(E,A)=\frac{b}{2\cos(\gamma)}$ (see \eqref{mufunk}).
Since $\angle CEA= \pi-2\gamma \geq \pi/2$,
then  we have $d(x,A) \geq d(E,A)=\mu(E)$ for any point $x\in [E,C]$. Moreover, we have $d(x,C) \geq d(E,C)=\mu(E)$ for any $x\in [E,B]$.
Therefore, $N_1:=\min\limits_{x\in [B,C]} \mu(x)=\mu(E)=\frac{b}{2\cos(\gamma)}$.

Analogously, $d(F,C)=d(F,B) \geq d(E,A)$ (the midpoint of $[A,C]$ is on the line segment $[E,C]$ due to the inequality $\angle CBA =\beta \leq \alpha \leq \pi/2$).
Hence $\mu(F)=d(F,C)=d(F,B)=\frac{a}{2\cos(\gamma)}$.
Since $\angle CFB= \pi-2\gamma \geq \pi/2$,
then for any point $x\in [F,C]$ we have $d(x,B) \geq d(F,B)=\mu(F)$. Moreover, for any $x\in [F,A]$ we have $d(x,C) \geq d(F,C)=\mu(F)$.
Therefore, $N_2:=\min\limits_{x\in [A,C]} \mu(x)=\mu(F)=\frac{a}{2\cos(\gamma)}$.

Now, we are going to find $N_3:=\min\limits_{x\in [A,B]} \mu(x)$. Since $\gamma \leq \pi/4$ and $\alpha \leq \pi/2$, we have $\alpha\geq \beta=\pi-\alpha-\gamma\geq \pi/4$.
Hence, we can apply Lemma \ref{lem1} for the side $[A,B]$ of the triangle $P=ABC$. Therefore, $N_3$ is equal to $a \sin(\beta)=b \sin (\alpha)$,
the length of the altitude of the triangle $P$ through the vertex $C$.

Note that $\delta(P)=\min\limits_{x\in P} \mu(x)=\min \{N_1, N_2, N_3\}$.
Since $a\geq b$, we have $N_2=\frac{a}{2\cos(\gamma)}\geq \frac{b}{2\cos(\gamma)}=N_1$, and
since $\gamma \leq \pi/4$, we get $\cos(\gamma)\geq \cos(\pi/4)=1/\sqrt{2}$. Since $\alpha \geq \pi/3$, we have $\sin(\alpha) \geq \sin(\pi/3)=\sqrt{3}/2$.
Further on, $2\sin(\alpha)\cos(\gamma)\geq 2 \cdot \frac{1}{\sqrt{2}} \cdot \frac{\sqrt{3}}{2} =\sqrt{3/2}>1$,
which implies
$N_3=b \sin (\alpha)>\frac{b}{2\cos(\gamma)}=N_1$. Therefore, we get $\delta(P)=\min\limits_{x\in P} \mu(x)=N_1=\frac{b}{2\cos(\gamma)}\,$.
\end{proof}

\begin{remark}\label{rem2}
If $\gamma \leq \pi/4$ and $\alpha = \pi/2$, then  $\beta=\pi/2-\gamma \geq \pi/4$ and
$\delta(P)=\frac{b}{2\cos(\gamma)}=\frac{b}{2\sin(\beta)}=\frac{a}{2\sin(\alpha)}=a/2$.
\end{remark}
\smallskip

From the above lemmas we immediately get {\bf the proof of Theorem \ref{theoechrt}}.
Now, we are going to use Theorem \ref{theoechrt} in order to get
a more simple proof of Theorem \ref{theo2wal}.

\smallskip

\begin{proof}[Proof of Theorem \ref{theo2wal}]
We are going to find all triangles $P$ with maximal value $L(P)/\delta(P)$, where $L(P)=a+b+c$ and $\delta(P)$ is used as shown in Theorem \ref{theoechrt}.
We have $L(P)=3a$ and $\delta(P)=\frac{\sqrt{3}}{2} \cdot a$ for any equilateral triangle $P$, hence, $L(P)/\delta(P)= 2\sqrt{3}$.
We will prove that $L(P)>2\sqrt{3}\cdot \delta(P)$ for any triangle $P$ that is not equilateral.

We consider three cases according to Theorem \ref{theoechrt}.

{\bf Case 1.} We suppose that $\alpha \geq \pi/2$ and, therefore, $\delta(P)=a/2$. Since  $b+c>a$, we get
$L(P)=a+b+c>2a=4\cdot \delta(P)>2\sqrt{3}\cdot \delta(P)$.
\smallskip

{\bf Case 2.} We suppose that $\gamma \geq \pi/4$ and, therefore, $\delta(P)=b\sin(\gamma)$.
Let $D$ be a point on the line segment $[B,C]$, such that $[A,D]$ is the altitude of $P$ through $A$.
$d(A,D)=b\sin(\gamma)=c\sin(\beta)$, $d(B,D)=c\cos(\beta)$, $d(C,D)=b\cos(\gamma)$. Hence,
\begin{eqnarray*}
\frac{L(P)}{\delta(P)}&=&\frac{a+b+c}{b\sin(\gamma)}=\frac{c\cos(\beta)+b\cos(\gamma)+b+c}{b\sin(\gamma)}=
\frac{b\cos(\gamma)+b}{b\sin(\gamma)}+\frac{c\cos(\beta)+c}{b\sin(\gamma)}\\
&&=
\frac{b\cos(\gamma)+b}{b\sin(\gamma)}+\frac{c\cos(\beta)+c}{c\sin(\beta)}=\frac{1+\cos(\beta)}{\sin(\beta)}+\frac{1+\cos(\gamma)}{\sin(\gamma)}=f(\beta)+f(\gamma),
\end{eqnarray*}
where $f(x)=\frac{1+\cos(x)}{\sin(x)}$. By direct computations we get $f^{\prime}(x)=-\frac{1+\cos(x)}{\sin^2(x)}=\frac{1+\cos(x)}{\cos^2(x)-1}=\frac{1}{\cos(x)-1}<0$ and
$f^{\prime \prime}(x)=\frac{\sin(x)}{(\cos(x)-1)^2}>0$ for $x\in (0,\pi)$.

Let us find the minimal value of the function $F(\beta,\gamma):=f(\beta)+f(\gamma)$ for $\beta\geq \gamma \geq \pi/4$, $\beta+\gamma \leq \frac{2}{3}\pi$.
Since $f^{\prime}(x)<0$, the minimal value can be achieved only when $\beta+\gamma=\frac{2}{3}\pi$ and $\gamma \in[\pi/4, \pi/3]$.

Let us consider the function $h(x):=f(x)+f\left(\frac{2}{3}\pi -x\right)$. It is convex due to the inequality $h^{\prime \prime}(x)=
f^{\prime \prime}(x)+f^{\prime \prime}\left(\frac{2}{3}\pi -x\right)>0$.
Hence, $h^{\prime}(x)=
f^{\prime}(x)-f^{\prime}\left(\frac{2}{3}\pi -x\right)$ is strictly increasing and $h^{\prime}(x)<h^{\prime}(\pi/3)=0$ for $x\in (0,\pi/3)$.

Therefore, the function $h$ achieves its minimal value on $[\pi/4, \pi/3]$ exactly at the point
$x=\pi/4$. This minimal value is $h(\pi/3)=2f(\pi/3)=2\sqrt{3}$.

Hence, $\frac{L(P)}{\delta(P)}=F(\beta,\gamma)>2\sqrt{3}$ for all $\beta$ and $\gamma$ such that $\beta\geq \gamma \geq \pi/4$, $\beta+\gamma \leq \frac{2}{3}\pi$,
and $(\beta,\gamma)\neq(\pi/3,\pi/3)$. For $\beta=\gamma=\pi/3$ we get the case of equilateral triangles.
\smallskip

{\bf Case 3.} We suppose that $\gamma \leq \pi/4$ and $\alpha \leq \pi/2$, therefore $\delta(P)=\frac{b}{2\cos(\gamma)}$.
We have $\beta \in [\pi/2-\gamma, \pi/2-\gamma/2]$ due to $\beta+\gamma =\pi-\alpha\geq \pi/2$ and $\gamma+2\beta\leq \gamma+\beta+\alpha=\pi$,
whereas $\gamma\in (0,\pi/4]$. We get $L(P)=a+b+c=\frac{b}{\sin(\beta)}\Bigl(\sin(\alpha)+\sin(\beta)+\sin(\gamma)\Bigr)$ and

\begin{eqnarray*}
\frac{L(P)}{\delta(P)}\!&=&\!\frac{2 \cos(\gamma)}{\sin(\beta)}\Bigl(\sin(\alpha)+\sin(\beta)+\sin(\gamma)\Bigr)=
\frac{2 \cos(\gamma)}{\sin(\beta)}\Bigl(\sin(\beta+\gamma)+\sin(\beta)+\sin(\gamma)\Bigr)\\
\!&=&\!\frac{2 \cos(\gamma)}{\sin(\beta)}\Bigl(\sin(\beta)\bigl(1+\cos(\gamma)\bigr)+\sin(\gamma)\bigl(1+\cos(\beta)\bigr)\Bigr)
\\
\!&=&\!2 \cos(\gamma)\bigl(1+\cos(\gamma)\bigr)+\sin(2\gamma)\cdot \frac{1+\cos(\beta)}{\sin(\beta)}=: G(\beta,\gamma).
\end{eqnarray*}
It is easy to see that $\frac{\partial G}{\partial \beta}(\beta,\gamma)=
-\sin(2\gamma)\cdot \frac{1+\cos(\beta)}{\sin^2(\beta)}=\frac{\sin(2\gamma)}{\cos(\beta)-1}<0$ for $\pi \geq \beta\geq \gamma >0$.
Therefore, the minimal value of $G(\beta,\gamma)$
for $\beta \in [\pi/2-\gamma, \pi/2-\gamma/2]$ and for a given $\gamma\in [0,\pi/4]$ can be achieved only when $\beta=\pi/2-\gamma/2$.
By direct computations we get
\begin{eqnarray*}
G(\pi/2-\gamma/2,\gamma)&=&2 \cos(\gamma)\bigl(1+\cos(\gamma)\bigr)+\sin(2\gamma)\cdot \frac{1+\sin(\gamma/2)}{\cos(\gamma/2)}\\
&=&2\cos(\gamma)+2\cos^2(\gamma)+4\bigl(1+\sin(\gamma/2)\bigr)\cos(\gamma)\sin(\gamma/2)\\
&=&4\cos(\gamma)\bigl(1+\sin(\gamma/2)\bigr)=g\bigl(\sin(\gamma/2)\bigr),
\end{eqnarray*}
where $g(t)=4(1-2t^2)(1+t)=4(1+t-2t^2-2t^3)$.
Since $0<\gamma/2 \leq \pi/8$, we get $0<t=\sin(\gamma/2)<\sin(\pi/8)=\sqrt{\frac{1}{2}\bigl(1-\cos (\pi/4)\bigr)}=\frac{\sqrt{2+\sqrt{2}}}{2}=:t_0$.

Since $g^{\prime\prime}(t)=-16(3t+1)$, $g$ is concave on the interval $[0,\infty)$. The
minimal value of $g(t)$ on the interval $[0,t_0]$ is achieved either at $t=0$, or at $t=t_0$.
Since $g(0)=4$ and $h(t_0)=2\sqrt{2}+\sqrt{4+2\sqrt{2}}>2\sqrt{3}$, we get that $\frac{L(P)}{\delta(P)}=G(\beta,\gamma)>2\sqrt{3}$
for all $\beta \in [\pi/2-\gamma, \pi/2-\gamma/2]$ and $\gamma\in (0,\pi/4]$.
The theorem is completely proved.
\end{proof}
\smallskip

\section{Convex curves and polygons of maximal perimeter}\label{sect.3}

A half-disk in the Euclidean plane $\mathbb{R}^2$ is a set which is isometric to
$$
HD(r)=\left\{x=(x_1,x_2)\in \mathbb{R}^2\,|\, x_2\geq 0, \, x_1^2+x_2^2 \leq r^2\right\}
$$
for some fixed $r>0$. The boundary of $HD(r)$ is the union of a half-circle of radius $r$ and a line segment of length $2r$.

\begin{lemma}\label{lem5}
Let $\Gamma$ be the boundary of some half-disk of radius $r$ in $\mathbb{R}^2$. Then $\delta(\Gamma)=r$ and $L(\Gamma)=(2+\pi) \cdot \delta(\Gamma)$,
where $L(\Gamma)$ means the perimeter of $\Gamma$.
\end{lemma}

\begin{proof}
The equality $\delta(\Gamma)=r$ follows from Proposition \ref{semcircle}. The second statement is obvious.
\end{proof}

\begin{theorem}\label{maxgeneral}
For any closed convex curve  $\Gamma$ in the Euclidean plane,  one has
$$
L(\Gamma)\leq (2+\pi) \cdot \delta(\Gamma),
$$
with equality exactly for boundaries of half-disks.
\end{theorem}

\begin{proof}
Let $\Gamma$ be a convex curve and let $o\in \Gamma$ be an extremal point for $\delta(\Gamma)$.
This means that $\Gamma \subset \{ x\in \mathbb{R}^2\,|\, d(x,o)\leq r\}=:D$, where $r=\delta(\Gamma)$.
Moreover, $\Gamma$ is a convex curve, and therefore there is a straight line $l$ through the point $o$
in $\mathbb{R}^2$ such that $\Gamma$ is situated in one of the half-planes determined by $l$.
Let us denote this half-plane by $H(l)$. Therefore, $\Gamma$ is a subset of the half-disk $HD:=D\cap H(l)$ of radius $r$.

By the monotonicity of the perimeter of convex curves (see Proposition \ref{monotper}) we get that the perimeter $L(\Gamma)$ of $\Gamma$ is less than or equal to
$(2+\pi)\cdot r$, the perimeter of the boundary of the semi-disk $HD$.
Moreover, equality holds if and only if $\Gamma$ is the boundary of $HD$. The theorem is proved.
\end{proof}
\smallskip

Now we are going to find all convex $n$-gons $P$ ($n\geq 2$) of maximal perimeter among all convex $n$-gons with the same value of $r=\delta(P)$.
At first, we consider an explicit construction of a special family of convex $n$-gons $U_n$ for $n\geq 2$.
\smallskip

Let us consider a regular $2(n-1)$-gon $P_n$, inscribed in a circle of radius $r>0$.
Take  points $A,B \in P_n$ that are opposite vertices of this polygon (i.e., $d(A,B)=2r$).
Now, consider one of the two half-planes determined by the straight line $AB$, say $H$, and consider the union $U_n$ of the line segment $[A,B]$
with the polygonal line $P_n \cap H$. This is an $n$-gon inscribed in the boundary of the half-disk $\{x\in \mathbb{R}^2\,|\, d(x,o)\leq r\}\cap H$,
where $o$ is the midpoint of the line segment $[A,B]$, see Fig.~\ref{Fig2}.
We have $\delta(U_n)=r$ by Proposition~\ref{semcircle}.
For $n=2$ we see that $P_2=U_2$ is a line segment of length $2r$.

\begin{figure}[t]
\begin{minipage}[h]{0.25\textwidth}
\center{\includegraphics[width=0.92\textwidth]{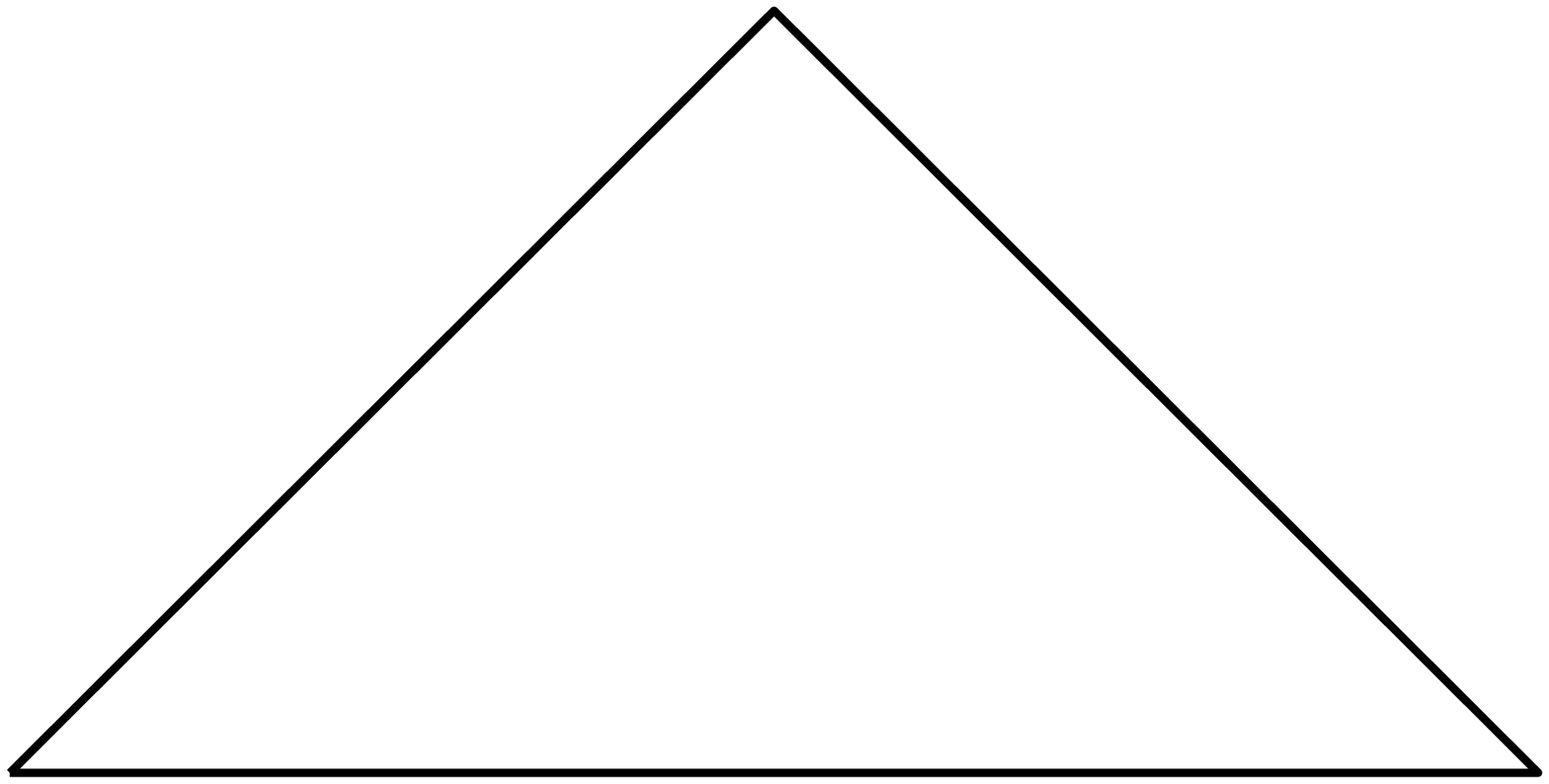}} a) \\
\end{minipage}
\quad\quad
\begin{minipage}[h]{0.25\textwidth}
\center{\includegraphics[width=0.99\textwidth]{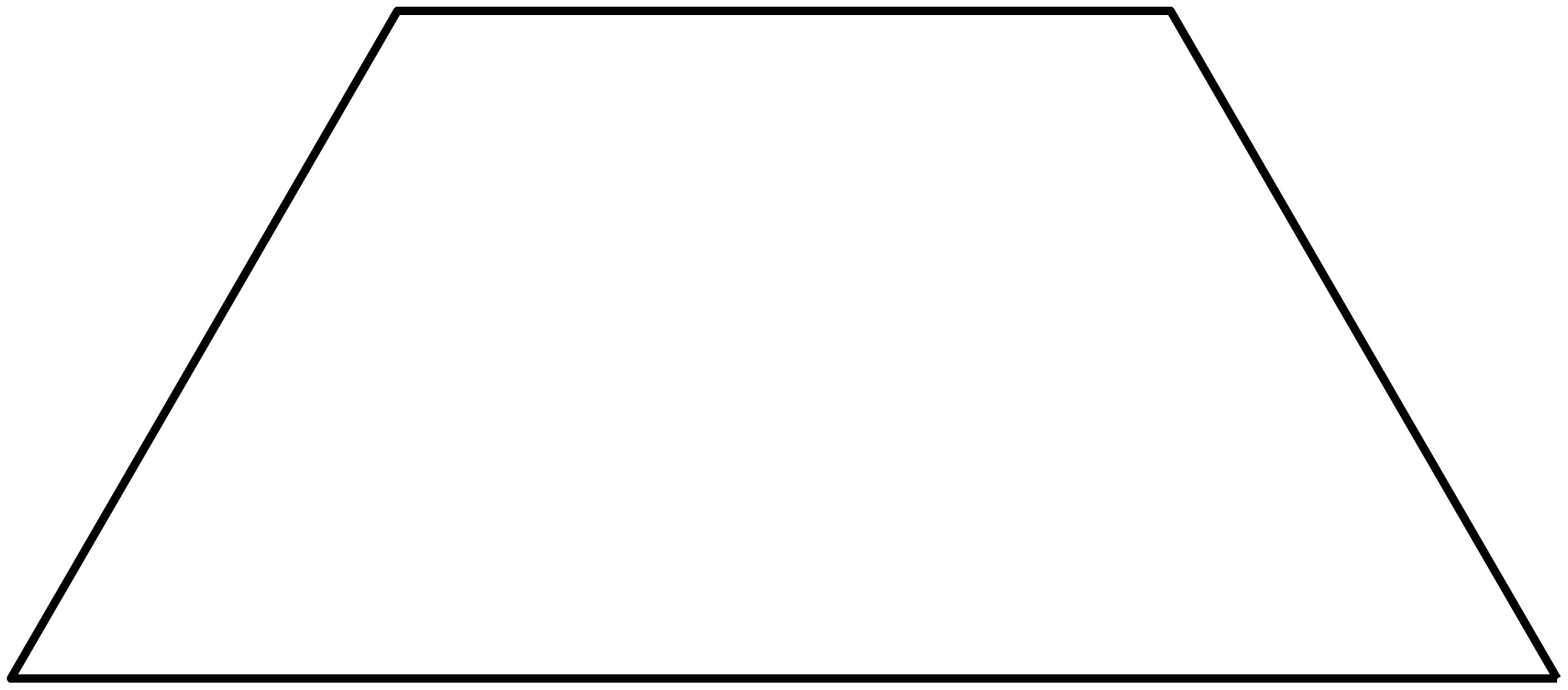}} b) \\
\end{minipage}
\quad\quad
\begin{minipage}[h]{0.25\textwidth}
\center{\includegraphics[width=0.90\textwidth]{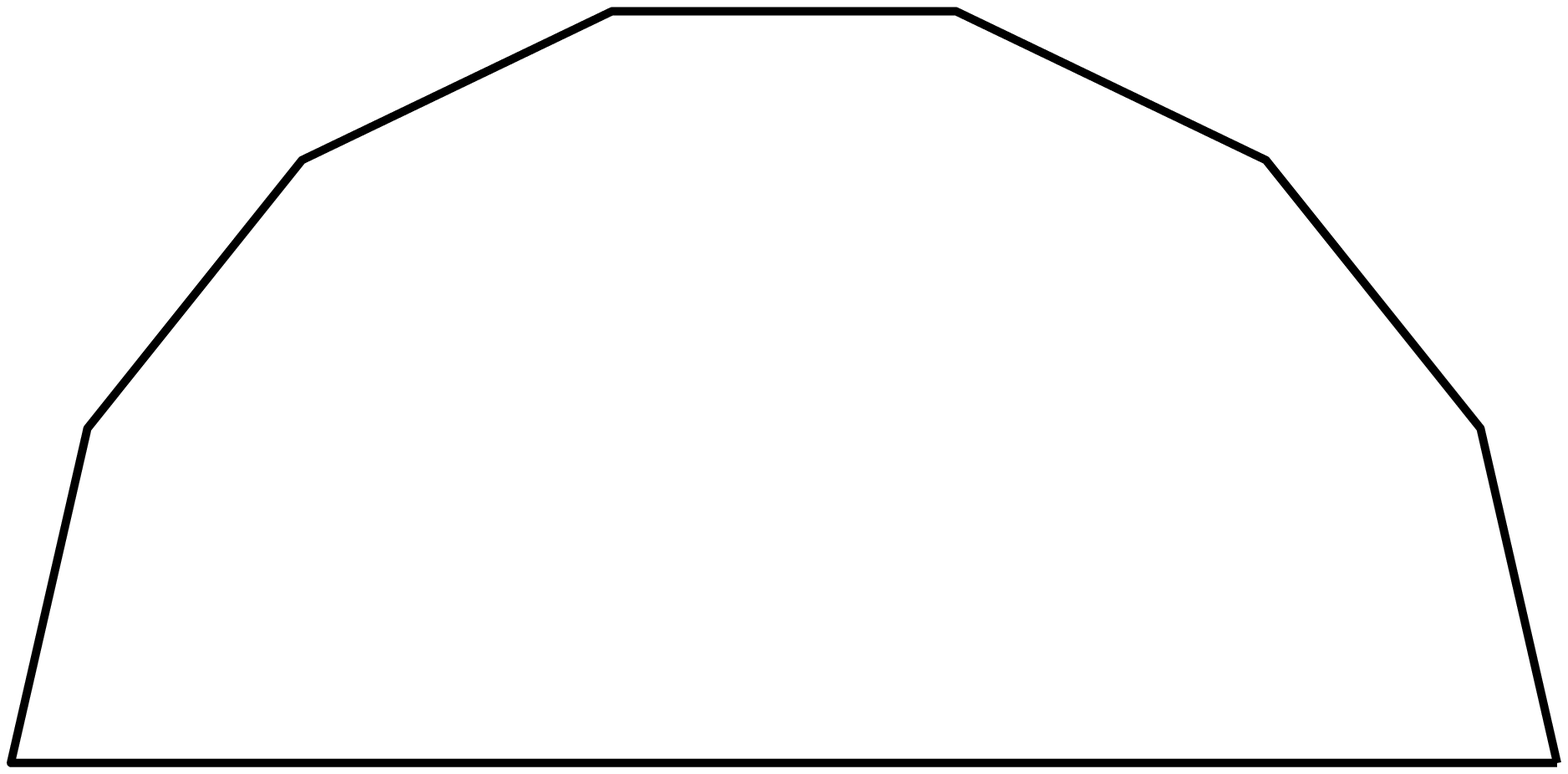}} c) \\
\end{minipage}
\caption{The polygons $U_n$ for: a) $n=3$; b) $n=4$; c) $n=8$.}
\label{Fig2}
\end{figure}

The perimeter $L(U_n)$ of $U_n$ is equal to $\lambda_n\cdot r$,  where
$\lambda_n=2\left(1+(n-1)\sin \bigl(\frac{\pi}{2(n-1)}\bigr)\right)$ for $n\geq 2$.
Note that $\lambda_2=4$, since $L(U_2)$ is the double length of the line segment $U_2$.
It is also easy to see that $\lambda_3=2(1+\sqrt{2})$, $\lambda_4=5$, and $\lambda_n \to 2+\pi$ as $n \to \infty$.

\begin{lemma}\label{lem6}
The above defined sequence $\{\lambda_n\}$, $n\geq 2$, is strictly increasing.
\end{lemma}

\begin{proof}
Let us consider the function $\psi(x)=\frac{\sin (x)}{x}$ for $x>0$.
We have $\psi^{\prime}(x)=\frac{\theta(x)}{x^2}$, where $\theta(x)=x\cos(x)-\sin(x)$. Since $\theta^{\prime}(x)=-x\sin(x)<0$ for $x\in (0,\pi)$, then
$\theta(x)<\theta(0)=0$ for all $x\in (0,\pi)$. This means that $\psi^{\prime}(x)<0$ and $\psi(x)$ strictly decreases on the interval $(0,\pi)$.
In particular, $\psi\left(\frac{\pi}{2(n-1)} \right)<\psi\left(\frac{\pi}{2n} \right)$ for all $n\geq 2$, which is equivalent to $\lambda_n <\lambda_{n+1}$.
\end{proof}

\begin{theorem}\label{maxpolygon}
For any convex $n$-gon  $P\subset \mathbb{R}^2$, $n\geq 2$, one has
$$
L(\Gamma)\leq 2\left(1+(n-1)\sin \left(\frac{\pi}{2(n-1)}\right)\right) \cdot \delta(\Gamma),
$$
with equality exactly for the $n$-gon $U_n$ defined above.
\end{theorem}

\begin{proof}
For $r>0$ and $n \geq 2$, we denote by $\mathcal{K}_n(r)$ the set of all convex polygons $P$ in the Euclidean plane with $\delta(P)=r$ and having
no more than $n$ vertices. In particular, polygons with 2 vertices are line segments, and we repeat that the perimeter of such a $2$-gon is assumed to be equal to its double length
(and the perimeter is continuous functional on  $\mathcal{K}_n(r)$ for every $n \geq 2$).
Let $\mathcal{M}_n(r)$ be the supremum of all perimeters of polygons from the set $\mathcal{K}_n(r)$.
Let us prove that there is a polygon in $\mathcal{K}_n(r)$ whose perimeter equals $\mathcal{M}_n(r)$.

Take a sequence $\{P_k\}$, $k\in \mathbb{N}$, $P_k\in \mathcal{K}_n(r)$, such that $L(P_k) \to \mathcal{M}_n(r)$ as $k \to \infty$.
Without loss of generality, we may suppose that all $P_k$ are situated in some ball of radius $r$
(we can move each $P_k$ in such a way that an extremal point for $\delta(P_k)$ coincides with the center of a given ball).
Using compactness arguments, we may assume that $P_k\to P_0$ (in the Hausdorff distance) for some $P_0 \in \mathcal{K}_n(r)$ as $k\to \infty$, hence $L(P_0)=\mathcal{M}_n(r)$.

Now we are going to prove that $P_0$ is isometric to the $n$-gon $U_n$ as above.
For this goal we will use induction on $n$. For $n=2$, $P_0$ is a line segment of length $2r$, that is $U_2$.
Suppose that  the assertion is true for all $k<n$ and prove it for $n$.

First note that $P_0$ is an $n$-gon. Indeed, if
$P_0$ is an $m$-gon for some $m<n$, then $\lambda_m = \mathcal{M}_m(r)\geq L(P_0)$ by the induction assumptions.
On the other hand, we know that $\mathcal{M}_n(r) \geq L(U_n)=\lambda_n$, but $\lambda_n>\lambda_m$ by Lemma \ref{lem6}.
Therefore, we get $\mathcal{M}_n(r)>L(P_0)$, and this contradiction proves that $P_0$ is an $n$-gon.

Now we are going to show that $P_0$ is isometric to $U_n$.
Let $o\in P_0$ be an extremal point for $\delta(P_0)$.
This means that $P_0 \subset \{ x\in \mathbb{R}^2\,|\, d(x,o)\leq r\}=:D$.
Moreover, $P_0$ is a convex $n$-gon, and therefore we have a straight line $l$ through the point $o$
in $\mathbb{R}^2$ such that $P_0$ is situated in one of the half-planes determined by $l$.
Let us denote this half-plane by $H(l)$. Therefore, $P_0$ is a subset of the half-disk $HD:=D\cap H(l)$ of radius~$r$.

Note that all vertices of $P_0$ are situated in the boundary $\partial(HD)$ of the half-disk $HD$.
Indeed, suppose the contrary. Take a point $S$ inside of $P_0$ and consider the rays $SA_i$, $i=1,...,n$, where $A_i$ are the vertices of $P_0$.
If some point $A_i$ is not in $\partial(HD)$ then we can modify $P_0$ into a polygon $P_0^{\prime}$, replacing the vertex $A_i$ by some point
$A_i^{\prime}$ lying on the ray $SA_i$
such that $d(S,A_i^{\prime})>d(S,A_i)$ and $A_i^{\prime}$ is taken so close to $A_i$ that $A_i^{\prime}\in HD$ and $P_0^{\prime}$ is convex.
By the monotonicity of the perimeter of convex curves (Proposition~\ref{monotper}), we get that  $L(P_0^{\prime})>L(P_0)=\mathcal{M}_n(r)$, which is
not true. This contradiction shows that all vertices of $P_0$ are situated in the boundary $\partial(HD)$ of $HD$.

Now let us consider vertices of $P_0$ that are in the line $l$. Since $o\in l\cap P_0$, there is at least one such vertex.
On the other hand, it is impossible to have three or more such vertices (in this case $P_0$ is an $m$-gon with some $m<n$).
Hence we have two possibilities: 1) only one vertex of $P_0$ (that should coincide with the point $o$) is in $l$, 2) exactly two vertices of $P_0$ are in $l$.
Fortunately, the first possibilities could be easily reduced to the second one. Indeed, one can rotate the polygon $P_0$ around the point $o$
so that it still remains in the half-disk $HD$, but its two vertices will already be on the line $l$.
Hence, without loss of generality, we may assume that exactly two vertices (say $A_1$ and $A_2$) of $P_0$ are in $l$.
Let us prove that $d(A_1,A_2)=2r$, i.e., $d(o,A_1)=d(o,A_2)=r$.
Suppose the contrary and modify the polygon $P_0$ into a polygon $P_0^{\prime}$, replacing the points $A_1$ and $A_2$ with two points
$A_1^{\prime}, A_2^{\prime} \in l$ such that $d(A_1^{\prime},o)=d(A_2^{\prime},o)=r$. It is clear that $P_0$ is inside of $P_0^{\prime}$,
hence (by Proposition \ref{monotper}) we have $L(P_0^{\prime})>L(P_0)=\mathcal{M}_n(r)$, which is
not true. Hence, $d(A_1,A_2)=2r$.

The last step is to prove that the rays emanating from the point $o$ successively to all the vertices of the polygon $P_0$
divide the semicircle
$\{ x\in \mathbb{R}^2\,|\, d(x,o)= r\} \cap H(l)$
into equal arcs. Indeed, this will imply that $P_0$ is isometric to $U_n$.

Now it suffices to prove the equality $d(A_{i-1},A_i)=d(A_{i+1},A_i)$ for three successive vertices $A_{i-1}$, $A_i$, $A_{i+1}$ of $P_0$.
Let us take a point $B$ on the arc between $A_{i-1}$ and $A_{i+1}$ and consider the polygon $P_0^{\prime}$ which is obtained from $P_0$
by replacing the point $A_i$ with~$B$. By construction of $P_0$ we have $L(P_0)\geq L(P_0^{\prime})$, equivalent
to the inequality $d(A_{i-1},A_i)+d(A_{i+1},A_i)\geq d(A_{i-1},B)+d(A_{i+1},B)$.
Hence $A_i$ is the point that gives the maximal value of the function $B\mapsto d(A_{i-1},B)+d(A_{i+1},B)$.

Put $\varphi:=\frac{1}{2} \angle A_{i-1}oA_{i+1}$ and $\psi:=\frac{1}{2}\angle A_{i-1}oB$.
Then
$$
d(A_{i-1},B)+d(A_{i+1},B)=2r\Bigl(\cos(\psi)+\cos(\varphi-\psi)\Bigr).
$$
Note that $\varphi \in (0,\pi/2]$ and $\psi\in [0,\varphi]$.
The minimal value of the function $h(\psi)=\cos(\psi)+\cos(\varphi-\psi)$ on the interval $[0,\varphi]$
is attained at the point $\psi=\varphi/2$, since
$h^{\prime}(\psi)=\sin(\varphi-\psi)-\sin(\psi) >(<)\,0$ for $\psi \in [0,\varphi/2)$ (for $\psi \in (\varphi/2, \varphi]$).
This means that $d(A_{i-1},A_i)=d(A_{i+1},A_i)$. The theorem is completely proved.
\end{proof}

\section{Final remarks}

Note that there is no result for $n$-gons, $n \geq 4$, similar to
Theorem \ref{theo2wal} for triangles.
We do not even know quadrangles that have the smallest perimeter among all convex quadrangles with a given value of \eqref{chebir1}.
It is easy to check that squares have no such property. In \cite{Walter2017}, it is conjectured that
$L(P)\geq \frac{4}{3}\sqrt{2\sqrt{3}+3}\,\cdot \delta(P)$
for any convex quadrangle $P\subset \mathbb{R}^2$.
Note that this inequality becomes an equality for quadrangles $P$ called ``magic kites'' (This notion is
taken from \cite{Walter2017} and means
convex quadrangles which are hypothetically extreme with respect to
the Chebyshev radius.)
Up to similarity, such a quadrangle could be represented by its vertices, that are as follows:
$$
(-1, 0), \quad (1, 0), \quad
\left(0, \frac{\sqrt{3}}{3}\sqrt{2\sqrt{3}+3}\right), \quad
\left(0, -\frac{1}{3}\sqrt{2\sqrt{3}-3}\,\right).
$$
Note also that  we have $L(P)=\frac{8}{\sqrt{5}}\,\cdot \delta(P)$ for a square $P$ and
$\frac{8}{\sqrt{5}}>\frac{4}{3}\sqrt{2\sqrt{3}+3}\approx 3,389946$.

It would be interesting to develop useful tools for the computation of
the relative Chebyshev radius \eqref{chebir1} in the case of quadrangles as well as  to obtain a result similar to Theorem \ref{theoechrt} for triangles.

We propose one more problem which is close to those considered above.

\begin{problem}
Given real number $l>0$ and a natural number $n \geq 3$, determine the best possible constant $C(n,l)$ such that the inequality
$L(P) \geq C(n,l)$ holds for every convex polygon $P$ with $n$ vertices and with the following property:
for the endpoints $A$ and $B$ of every side of $P$, there is a point $C\in P$ such that $d(A,C)=d(B,C) \geq l$.
\end{problem}

It is easy to show that $C(3,l)=3l$. On the other hand, the answer for $n\geq 4$ is unknown.
It is also interesting to study the analogous problem by
taking, instead of the
perimeter $L(P)$, the area $S(P)$ of the polygon $P$.

\vspace{10mm}

\end{document}